%% file: partialnfold_final.tex
\pdfoutput=1

\documentclass[oneside]{article}   	
\usepackage{geometry}                		
\usepackage{graphicx}				

\usepackage{longtable}
\usepackage{array}
\usepackage{hyperref}
\usepackage{amssymb}
\usepackage{amsmath}
\usepackage{physics}    
\usepackage{tikz-cd}    
\usepackage{graphicx}
\usepackage{mathtools}
\usepackage{amsthm}
\usepackage{enumerate}
\usepackage{bbm}
\usepackage{subfiles}
\usepackage{xcolor}

\newcommand{\R}{\mathbb{R}}

\newcommand{\Z}{\mathbb{Z}}

\newcommand{\Prob}{\mathbb{P}}
\newcommand{\Expe}{\mathbb{E}}

\newcommand{\calG}{\mathcal{G}}

\newcommand{\calM}{\mathcal{M}}

\newcommand{\calP}{\mathcal{P}}

\newcommand{\rmx}{\mathrm{x}}

\newcommand{\bigA}{\mbox{\textit{\Large A}}}
\newcommand{\rvline}{\hspace*{-\arraycolsep}\vline\hspace*{-\arraycolsep}}

\DeclareMathOperator{\GL}{GL}

\DeclareMathOperator{\Var}{Var}
\DeclareMathOperator{\Cov}{Cov}

\theoremstyle{plain}
\newtheorem{theorem}{Theorem}[section]
\newtheorem{lemma}[theorem]{Lemma}
\newtheorem{corollary}[theorem]{Corollary}

\newtheorem{proposition}[theorem]{Proposition}

\theoremstyle{definition}

\newtheorem*{problem*}{Problem}

\theoremstyle{remark}

\title{Asymptotics of the partial \(n\)-fold dimer model} 
\author{Christina Meng\footnote{Department of Mathematics, Yale University}}
\date{September 28, 2025}							

\begin{document}

\maketitle

\input{partialnfold_abstract_final}

\input{partialnfold_intro_final}

\input{partialnfold_nmultiwebs_final}

\input{partialnfold_cycle_final}

\input{partialnfold_local_final}

\bibliographystyle{plain}
\bibliography{partialnfold}

\end{document}

%% file: partialnfold_abstract_final.tex

\begin{abstract}

We study a model of colored multiwebs, which generalizes the dimer model to allow each vertex to be adjacent to \(n_v\) edges. These objects can be formulated as a random tiling of a graph with partial dimer covers. We examine the case of a cycle graph, and in particular we describe the local correlations of tiles in this setting.

\end{abstract}

%% file: partialnfold_intro_final.tex

\section{Introduction}

Consider the scenario of a course in which each student has to sit an exam consisting of \(n\) problems, which are chosen from a collection of \(N\) problems written by the professor. Suppose the course has \(m\) students and \(k\) teaching assistants, and the problems are assigned using the following procedure. Each TA chooses a subset of \(l\) problems and then partitions them amongst the students, so that in the end each student receives \(n\) distinct problems. In how many ways can such an assignment be done?
In this paper we study a version of the dimer model that will answer this question. 

The dimer model constitutes the study of probability spaces of dimer covers, or perfect matchings, on a graph.
This statistical mechanics model has proved to have a multitude of interesting connections to other areas of mathematics.
A classical result of Kasteleyn is the following: one can enumerate weighted dimer covers of a planar graph by computing the determinant of the associated Kasteleyn matrix \cite{K63}. 

These \(n\)-fold dimer covers, or \(n\)-multiwebs, recently appeared in the representation theory literature in \cite{FLL19}, and works \cite{DKS24} and \cite{KS23} make sense of the \(n\)-fold model where the appropriate analogue to edge weights on graphs are \(\GL_n\)-connections.
The results above are specific to planar graphs, and little is known once we relax the planarity condition. 
A variant of the model was introduced in \cite{KP22}, allowing for asymptotic analysis.

The objects of interest in this paper are \((n_v)_v\)-multiwebs with edges colored with \(N\) colors for \(N\geq n_v\). We study multiwebs with \(N\) large as random objects. 
In Section \ref{sec:cnmultiwebs} we will define the model, and present an equivalent formulation of it as tilings using partial dimer covers, with probability measure coming from some initial tile weights. 
We apply the relevant tools from \cite{KP22}, which studied the related, yet distinct, multinomial tiling model.
We then compute the growth of the partition function in the limit as the vertex multiplicities \(n_v\) and the number of colors \(N\) become large (Section \ref{subsec:growthrate}). In this limit we also obtain a Gaussian approximation of the number of occurrences \(X_t\) of a tile of type \(t\) in a random multiweb. In particular, the covariance matrix of \(\vb{X} = (X_t)_t\) can be expressed in terms of linear maps depending on the data of the graph as well as the critical weight function on tiles, which come from a specific measure-preserving transformation of the original weights (Section \ref{subsec:Xt}).

To study the asymptotic version of the model, we specify vertex densities \(\alpha_v\) at each vertex \(v\) such that the fraction \(\frac{n_v}{N}\) of tiles covering \(v\) tends to \(\alpha_v\) as \(n_v, N\to \infty\).
In the case of an odd-length cycle graph with uniform vertex densities, we study the model somewhat concretely. Proposition \ref{prop:redtilingpoly} presents a closed-form expression for the reduced tiling polynomial of a cycle graph.
There turns out to be a unique critical vertex density for which the critical tile weights are uniform. 
Theorem \ref{thm:alphahat} computes this density, and demonstrates that it maximizes the growth rate of the partition function with respect to the growth of \(n_v\) and \(N\). 
Fixing the density to be critical, we compute the distribution of \(\vb{X}\). 
We have concrete expressions for the mean and covariance; Figure \ref{fig:L=9covX} is a numerical picture of the covariance of \(\vb{X}\) for the cycle graph with nine vertices. 

\begin{figure} 
\centering
\includegraphics[scale=0.6]{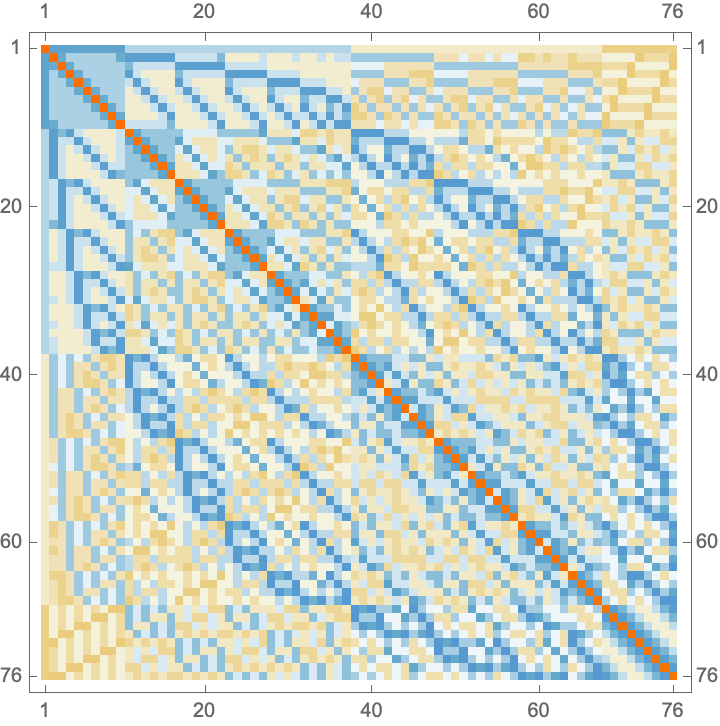}
\caption{Scaled covariance matrix \(\frac{1}{N} \Cov(\vb{X})\) for \(L=9\) cycle graph, with colors from blue to white to orange indicating negative, zero, positive values, respectively. The tiles are ordered by their size and then lexicographically in the vertex indices.}
\label{fig:L=9covX}
\end{figure}

We also describe the local correlations of tiles on a small segment of the cycle graph. Specifically, we partition the set of tiles according to which edges in a fixed window of five vertices they contain, and we examine the number of occurrences \(S_j\) of each equivalence class \(j\) in a multiweb. We then describe the joint distribution of \(\vb{S} = (S_j)_j\) in the limit; Figure \ref{fig:L=31covS} is a numerical picture for the covariance matrix of \(\vb{S}\) for the cycle graph with 31 vertices. Finally, we observe that there is a limiting mean and variance of \(S_j\) as the size of the graph \(L\) tends to infinity (Section \ref{subsec:limlocal}).

\begin{figure} 
\centering
\includegraphics[scale=0.4]{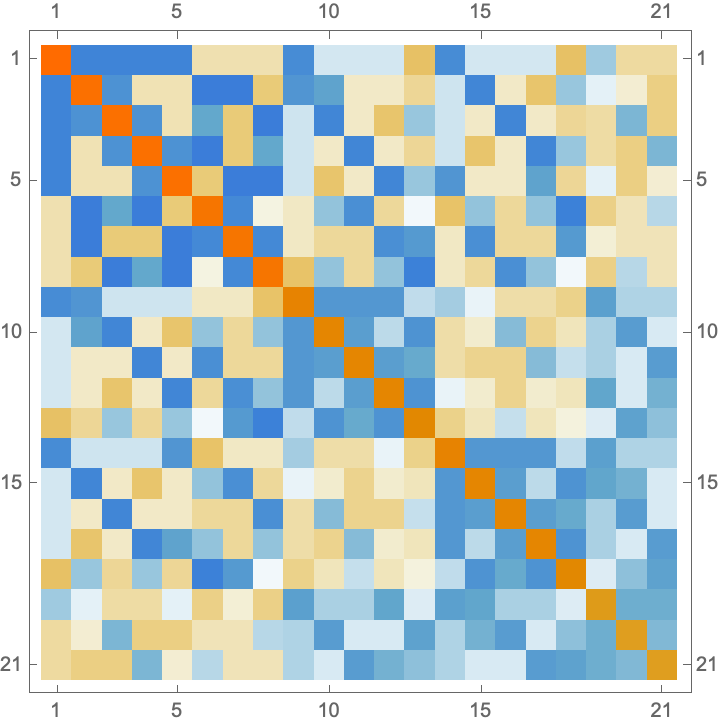} 
\caption{Scaled covariance matrix \(\frac1N \Cov(\vb{S})\) for \(L=31\) cycle graph.  }
\label{fig:L=31covS}
\end{figure}

\paragraph{Acknowledgements.}

We thank Richard Kenyon for helpful discussions. This research was supported by NSF grant DMS-1940932 and the Simons Foundation grant 327929.

%% file: partialnfold_nmultiwebs_final.tex

\section{Colored multiwebs} \label{sec:cnmultiwebs}

In this section we define the model and compute the exponential growth rate of the partition function. We also describe the distribution of tile multiplicities in a random multiweb.

Let \(\calG = (V,E)\) be a finite graph. Fix nonnegative integer multiplicities \(\vb{n} = (n_v)_{v}\) at each vertex \(v\), and suppose we have \(N\) colors. Write \([N]\) for the set of colors, and we use \(\calP(\cdot)\) to denote a power set. A \textit{colored \(\vb{n}\)-multiweb} is an assignment of subsets of colors to each edge \(m\colon E \to \calP([N])\) such that no colors are repeated among edges sharing an endpoint, and such that the total number of colors incident to each vertex \(v\) is \(n_v\). That is to say, we require that \(m(u_1 v) \cap m(u_2 v) = \varnothing\) for vertices \(u_1\neq u_2\) and \(\sum_{u \sim v} \abs{m(uv)} = n_v\) for each vertex \(v\). This definition of colored multiweb is more general than the one in \cite{KS23}, which requires that every vertex have multiplicity \(n\) and there be exactly \(n\) colors. Write \(\Omega_{\vb{n},N}(\calG)\) for the set of colored \(\vb{n}\)-multiwebs. 

A \textit{partial dimer cover}, or a \textit{partial matching}, of \(\calG\) is a subset of edges that covers each vertex either zero or one time(s). We obtain a \textit{dimer cover} in the special case where every vertex is covered exactly once. 
By the disjointness condition above, each color defines a partial matching on \(\calG\). Precisely, if we fix \(m\in \Omega\), the subset of edges \(\{e: k\in m(e)\}\) is a partial matching for any color \(k\). Let the set of tiles \(T = T(\calG) \subset \calP(E)\) be the set of partial matchings of \(\calG\). Then, we can alternatively define a colored \(\vb{n}\)-multiweb as \(m\colon[N]\to T\) satisfying the condition that each vertex \(v\) is covered \(n_v\) times. 
Here we use the term "tiles" somewhat unconventionally since they are not necessarily connected and the "tilings" have some multiplicity at each site.


The tiling formulation of our model is amenable to being studied via the techniques of \cite{KP22}.  
Our model is similar to theirs in that valid configurations are those with the correct vertex multiplicities, and the measure comes from a tile weight function. However, the multinomial model considers tilings of a "blowup graph", and there is no notion of colors, so in particular the tiles comprising a tiling are not labelled or ordered in any way.
In our case, increasing \(N\) increases the entropy of the measure we will define below.
In what follows we will review the relevant results of \cite{KP22}, specifying how with slight modifications, their results apply to our model. 

\subsection{Probability measure} \label{subsec:probmeas}

Introduce a weight function on tiles \(w\colon T \to \R_{>0}\). Let \(m_t\) be the number of times a tile of type \(t\) appears in a multiweb \(m\), and define the weight of \(m\) to be \(w(m) = \prod_t w(t)^{m_t}\). Now, the measure \(\Prob = \Prob_{w, \vb{n}, N}\) we would like to study is the one satisfying \(\Prob(m) \propto w(m)\). The partition function \(Z = Z_{w, \vb{n},N}(\calG)= \sum_{m\in \Omega}w(m)\) is the required normalization factor, and so \(\Prob(m) = \frac{w(m)}{Z}\). 

Next, index a collection of variables by the vertices: \(\vb*{x} = (x_1, \dots, x_V)\); here, we also use \(V\) to refer to the cardinality of this set.
Let \(t_v\) be \(1\) if vertex \(v\) is covered by \(t\) and zero otherwise.
The \textit{tiling polynomial} is
\[P(\vb*{x}) = \sum_{t\in T} w(t) \prod_{v} x_v^{t_v}.\]
Later it will be convenient to allow a vertex to occur in a tile with some multiplicity \(t_v \in \Z_{\geq 0}\), in which case the tiling polynomial is still defined as above. 

Let \(\vb*{x}^{\vb{n} }= \prod_v x_v^{n_v} \). One can then interpret the \(N\)-th power of \(P(\vb*{x})\) as the generating function for all tilings with \(N\) partial matchings:
\[\sum_{\vb{n} \geq \vb{0}} Z_{w, \vb{n},N} \vb*{x}^{\vb{n}} = \pqty{P(\vb*{x})}^N,\]
where the sum is over all nonnegative multiplicity vectors. 
Now, the coefficient of the \(\vb*{x}^{\vb{n}}\) term is precisely the desired partition function:
\[Z_{w,\vb{n},N} = [\vb*{x}^{\vb{n}}] \pqty{P(\vb*{x})}^N.\]
Note it is possible for \(\Omega_{\vb{n},N}\) to be empty for some \(\vb{n}\), in which case \(Z_{w, \vb{n},N}\) would be zero; we consider which multiplicities are feasible in Section \ref{subsec:growthrate}.

To discuss the asymptotics of the model, we will need the following notion of equivalence between tile weight functions. 
Two weight functions \(w, w'\colon T\to \R_{>0}\) are \textit{gauge equivalent} if there exists a function \(f\colon V\to \R_{>0}\) such that \(w'(t) = w(t) \prod_{v\in t} f(v)\). Such functions \(f\) are called \textit{gauge transformations}. 
For the lemma below, the exact same proof as that of Lemma 2.1 in \cite{KP22} works.

\begin{lemma}[\cite{KP22}]
Gauge equivalent tile weights give the same probability measure on \(\Omega_{\vb{n},N}\).
\end{lemma}

\subsection{Tiling Laplacian} \label{subsec:tilinglaplacian}

We define a Laplacian operator associated to the tiling problem. Firstly, for a set \(X\) we use exponential notation \(\R^X\) to refer to the real linear space with basis \((e_x)_{x\in X}\) (in some arbitrary fixed order).
The \textit{incidence map} \(D\colon \R^T \to \R^V\) of the tiling problem is the linear map \(D(e_t) = \sum_{v\in t} t_v e_v\). 
The matrix \(D = (D_{v,t})_{v,t}\) representing this linear map with respect to \((e_v)_v\) and \((e_t)_t\) is called the \textit{incidence matrix}. 
Next, identifying each basis element \(e_v\) with its dual gives \(\R^V \cong (\R^V)^*\), and similarly \(\R^T \cong (\R^T)^*\).
Then, the transpose of the incidence map is a linear map \(D^*\colon \R^V\to \R^T\).
We also define \(C\colon\R^T \to \R^T\) as the map which scales each basis element by the corresponding tile weight, i.e. \(C(e_t) = w(t) e_t\).
Then, the \textit{tiling Laplacian} \(\Delta\colon\R^V \to \R^V\) is the composition \(\Delta = DCD^*\). 

In general the incidence map \(D\) is neither injective nor surjective. Its domain and codomain decompose orthogonally as \(\R^V\cong \Im(D)\oplus \ker(D^*)\) and \(\R^T \cong \Im(D^*) \oplus \ker(D)\). 
Then, the restrictions \(D\colon \Im(D^*)\to \Im(D)\) and \(D^*\colon \Im(D)\to \Im(D^*)\) are isomorphisms, and so \(\Delta\colon\Im(D)\to \Im(D)\) is an isomorphism as well. We note that \(\ker(D^*)\) is the set of gauge transformations which leave tile weights unchanged, in the sense that having \(\prod_v f(v)^{t_v} = 1\) for every tile \(t\) is equivalent to satisfying \(D^\top \log f = \vec{0}\).

\subsection{Fixed tile sizes} \label{subsec:tilesizes}

It will be convenient to work with a homogeneous tiling polynomial since it simplifies our calculations in Section \ref{sec:excycle}. Let \(s(t)\) be the size of the partial matching \(t\), which we define to be the number of edges, or dimers, constituting \(t\). For each \(t\), multiply the corresponding monomial in \(P(\vb*{x})\) by \(x_0^{V-2s(t)}\). The resulting polynomial \(\hat{P}(\hat{\vb*{x}}) = \hat{P}(x_0,\dots, x_V)\) is then homogeneous of degree \(V\). 

It is still possible to interpret \(\hat{P}(\hat{\vb*{x}})\) as a tiling polynomial. The corresponding tiling problem has fixed tile sizes, and we claim that it is equivalent to the original problem. Consider a modified graph \(\hat{\calG}\) which has an additional zero vertex \(v_0\) and no additional edges; now, \(x_0\) is the vertex variable corresponding to \(v_0\). Viewing a tile as a multiset of vertices, we have a bijection \(h \colon T\to \hat{T}\) to the set of homogenized tiles given by 
\[t\mapsto t\cup \{\underbrace{v_0, \dots, v_0}_{V-2s(t)}\}.\]
Furthermore, since every tile is a multiset of \(V\) vertices, we have
\(NV = n_0 + \sum_v n_v\).
Thus, the multiplicity \(n_0\) of \(v_0\) is determined by \(n_1,\dots, n_V\). 
In summary, the mapping \(h\) induces a measure-preserving bijection \(\Omega_{\vb{n},N}(\calG)\to \hat{\Omega}_{\hat{\vb{n}},N}(\hat{\calG})\).
We will always assume we have equal tile sizes from here on, so we drop the hat notation as there is no ambiguity.

\subsection{Growth rate of partition function and critical gauge} \label{subsec:growthrate}

It is possible to approximate the partition function for large \(n_v\) and \(N\). As these parameters become large we would like the sequence of spaces of multiwebs \(\Omega_{\vb{n},N}\) to be nonempty for each \(n_v\) and \(N\).  
Specifically, if for a fixed number of colors \(N\) there exists a colored multiweb with vertex multiplicities \(\vb{n} \in \pqty{\Z_{\geq 0}}^V\), we refer to \(\vb{n}\) as \textit{feasible} with respect to this \(N\).
Define the affine subspace \(H_N = \{(s_1,\dots, s_T) \in \R^T: \sum_i s_i = N \}\). 
In terms of the incidence map \(D\), the set of feasible multiplicities when there are \(N\) colors is \(\calM_{\Z}^N \coloneq D\pqty{(\Z_{\geq 0})^T \cap H_N}\).

Now, fix a density \(\alpha_v\) at each vertex. Let \(n_v\to \infty\) and \(N\to \infty\) simultaneously, such that each multiplicity vector \(\vb{n}\) lies in \(\calM^N_\Z\) and \(\frac{n_v}{N} \to \alpha_v\). Then, \(\alpha_v\) is the asymptotic fraction of tiles covering \(v\). 
Again, not all choices of vertex densities \(\vb*{\alpha} = (\alpha_v)_v\) are possible. We say \(\vb*{\alpha} \) is \textit{feasible} if it is in \(\calM_\R \coloneq D\pqty{(\R_{> 0})^T \cap H_1 }\).
While in general the fractional multiplicities \(\pqty{\frac{n_v}{N}}_v\) lie in \(  D\pqty{(\R_{\geq 0})^T \cap H_1 }\) when \(\vb{n}\) is feasible, we will only consider \(\vb*{\alpha}\) in the interior. For \(\vb*{\alpha}\) on the boundary, the associated probability measure on tiles (coming from the critical tile weights described in Theorem \ref{thm:critgauge} below) is still well defined, but it degenerates in the sense that certain tiles occur with probability zero so we exclude this case.

The growth of the partition function \(Z\) is described by its \textit{exponential growth rate} \(\sigma = \sigma(w,\vb*{\alpha})\) in the sense that asymptotically,
\[Z \sim e^{\sigma N+ o(N)} \qas n_v, N\to \infty. \]
A theorem of \cite{KP22} presents an expression for the growth rate in terms of the tiling polynomial and a specific gauge transformation. 
We may apply the theorem in our case since for both models the partition function can be expressed in terms of a coefficient of a power of the tiling polynomial, and the proof largely consists of manipulating this polynomial.
However, the differences in the two models result in different combinatorics, so we have a modified expression for the growth rate in equation (\ref{eqn:growthratesigma}) below.

\begin{theorem}[\cite{KP22}] \label{thm:critgauge}
For any \(\vb*{\alpha} \in \calM_\R\) and weight function \(w\) there is a unique gauge-equivalent weight function \(w'\) with the property that for all \(v\) the sum of weights of tiles containing vertex \(v\) (counted with multiplicity) is \(\alpha_v\), that is 
\(\sum_{t} t_v w'(t) = \alpha_v\).
The corresponding gauge transformations \(f\colon V\to \R_{>0}\) are exactly those which solve the "criticality equations" 
\begin{equation} \label{eqn:crit}
\frac{x_v \partial_{x_v} P}{P} = \alpha_v
\end{equation}
with \(x_v = f(v)\). The growth rate of \(Z_{w,\vb{n},N}\) is 
\begin{equation}\label{eqn:growthratesigma}
\sigma(w, \vb*{\alpha}) \coloneq \lim_{n_v, N\to \infty} \frac1N \log Z_{w,\vb{n},N} = \log P(\vb{x}) - \sum_v \alpha_v \log(\rmx_v) 
\end{equation}
for any solution \(\vb{x}\) to (\ref{eqn:crit}).
\end{theorem}

As a consequence of this theorem we obtain a new weight function \(w'\) given by \(w'(t) = w(t) \prod_{v\in t} \rmx_v\).
We refer to \(w'\) as the \textit{critical weight function} and a positive solution \(\vb{x}\) to the criticality equations as a \textit{critical solution} or \textit{critical gauge}. We will see in Section \ref{subsec:Xt} how to obtain the joint distribution of tile multiplicities from the critical gauge.

The proof of the theorem demonstrates that as a function of \(\vb*{\alpha}\), \(\sigma\) is the negative of the Legendre dual of a strictly convex function, and hence \(\sigma\) is strictly concave. 
The proof also reveals that positive solutions to the criticality equations always exist but are not in general unique, though each solution results in the same critical weight function. In particular, \(\vb{x}\) is only unique up to a global multiplicative constant, as well as (pointwise) multiplication by gauge transformations \(f\) coming from \(\ker D^*\), in case this kernel is nontrivial. 

The non uniqueness of the critical solution allows us to normalize \(P(\vb{x})\) to arrive at the probability generating function of a related distribution. 
Since \(P\) is homogeneous we may scale \(\vb{x}\) to let \(P(\vb{x})=1\). In this case the tiling polynomial in the critical weights 
\begin{equation}\label{eqn:crittilingpoly}
\tilde{P}(\vb*{x})= \sum_t w'(t) \prod_{v} x_v^{t_v}
\end{equation}
is the probability generating function for choosing a tile of type \(t\) with probability \(w'(t)\).
Then, \(\tilde{P}(\vb*{x})^N\) is the probability generating function for choosing \(N\) independent identically distributed tiles, and by Theorem \ref{thm:critgauge}, this multinomial distribution has the property that the expected vertex multiplicity at \(v\) is \(N\alpha_v\).

\subsection{Distribution of tile multiplicities} \label{subsec:Xt}

Let \(X_t : \Omega_{\vb{n},N}\to \Z_{\geq 0}\) be the random variable counting the number of times a tile of type \(t\) occurs in a colored \(\vb{n}\)-multiweb, i.e. \(X_t(m) =m_t\). Section 5 of \cite{KP22} explains how to approximate the distribution of \(\vb{X} = (X_t)_t\) when \(n_v\) and \(N\) are large. 
The idea is firstly to take the limit of the multinomial distribution described in Section \ref{subsec:growthrate}, where the event probabilities are given by the critical tile weights from Theorem \ref{thm:critgauge}.
Then, condition the resulting multivariate Gaussian to ensure the correct vertex densities. 

Let the tile weight matrix \(C\) be the one corresponding to critical weights. Then, \(\Delta = D C D^*\) is the tiling Laplacian for critical weights.
Taking the limits \(n_v\to \infty\) and \(N\to \infty\) in the manner described in Section \ref{subsec:growthrate}, we have the following result regarding the distribution of \(\vb{X}\).

\begin{theorem}[\cite{KP22}] \label{thm:Xtdistr}
In the limit of large \(n_v\) and \(N\), the joint distribution of the \(X_t\) tends to a multidimensional Gaussian with mean \(\Expe(\vb{X}) = N \vec{w}\) (where \(\vec{w}\) is the vector of critical tile weights) and covariance matrix \(\Cov(\vb{X}) = NC(I-D^\top \Delta ^{-1} DC)\).  
\end{theorem}

Following the proof in \cite{KP22} line by line also proves the theorem for our model, except for one step where the justification has to be modified; we now describe how to do so.
As part of the proof we need to show that the random variables \(X_t\) counting the tile multiplicities satisfy
\(\Cov(X_t, X_{t'}) = \delta_{t t'} \Expe(X_t)+ \Expe(X_t) \left(\Expe^* (X_{t'}) - \Expe(X_{t'}) \right)\),
where the \(*\) superscript refers to the probability measure on \(\Omega_{\vb{n^*}, N-1}(\calG)\) with \(n^*_v = n_v -t_v\). 
Firstly, each \(X_t\) is a sum of \(\{0,1\}\)-valued random variables, each corresponding to one of the colors: \(X_t = X_t^1 + \cdots + X_t^N\). By symmetry, \(\Expe(X_t^i X_{t'}^j) = \Expe(X_t^1 X_{t'}^2)\) for all \(i\neq j\), and
\[\Expe(X_t^1 X_{t'}^2) = \Prob(X_t^1 =1, X_{t'}^2 =1) = \Prob (X_{t'}^2 =1 \mid X_t^1 = 1) \Prob (X_t^1 =1) = \Prob^* (X_{t'}^2 = 1)\Prob(X_t^1 =1).\]
Furthermore, for all \(i\) we have \(\Expe(X_t^i X_{t'}^i) = \Expe(X_t^1 X_{t'}^1)\), which is zero if \(t\neq t'\) since then \(t\) and \(t'\) cannot both be assigned to a single color. Lastly, we have \(\Expe^*(X_{t'}) = (N-1) \Expe^*(X_{t'}^2)\). Thus,
\[\Expe(X_t X_{t'})  =N \Expe(X_t^1X_{t'}^1) + N(N-1) \Expe(X_t^1 X_{t'}^2) = \delta_{t t'} \Expe(X_t)+ \Expe(X_t) \Expe^*(X_{t'}).\]

Note in general \(\Delta\) may not be invertible as it can have a nontrivial kernel \(\ker \Delta = \ker D^*\). However we saw in Section \ref{subsec:tilinglaplacian} that \(\Delta\colon \Im(D) \to \Im(D)\) is invertible, and restricting domains and codomains as follows does not change the overall map: 
\[\R^T \xrightarrow{C} \R^T \xrightarrow{D} \Im(D) \xrightarrow{\Delta^{-1}} \Im(D) \xrightarrow{D^*} \R^T.\]
Thus, it is always possible to make sense of the \((T\times T)\)-matrix \(D^\top \Delta ^{-1} DC\).

%% file: partialnfold_cycle_final.tex

\section{Cycle graph with constant vertex densities} \label{sec:excycle}

In this section we study the model on a cycle graph. We identify a special value for the vertex density which results in a uniform tile distribution. Lastly, we describe how to invert the tiling Laplacian; we will use it to compute local tile distributions in Section \ref{sec:exlocal}. 

Let \(\calG\) be the cycle graph of odd length \(L\geq 3\) (the case of \(L=1\) is trivial since it has a single partial matching, namely the empty matching with no edges). The vertex and edge sets of \(\calG\) are \(V = \{v_1,\dots, v_L\}\) and \(E = \{v_1 v_2, \dots, v_{L-1} v_L, v_L v_1\}\). Suppose our initial weight function is identically \(1\) and the vertex densities are constant, i.e. we have \(w(t) \equiv 1\) and \(\alpha_{v} \equiv \alpha\) for \(v = v_1,\dots, v_L\). It follows that \(\alpha_0 = L(1-\alpha)\). Here we note that for \((\alpha_0,\alpha,\dots,\alpha)\) to be feasible, we need \(\alpha\in \left(0,\frac{L-1}{L}\right)\). The maximum number of vertices covered by a partial matching of \(\calG\) is \(L-1\) since \(L\) is odd, so the density is bounded above by \(\frac{L-1}{L}\).

\subsection{Critical gauge}

By Theorem \ref{thm:critgauge}, a critical solution \(\vb{x}\) is a real positive solution to the system of equations (\ref{eqn:crit}). 
Recall from Section \ref{subsec:growthrate} that the critical solution is unique up to a global multiplicative factor and gauge transformations in \(\ker D^*\). 
The kernel is trivial in this case. If \(f\) is a gauge transformation not changing tile weights, then \(f(v_i) = \frac{1}{f(v_{i+1})}\) and so \(f(v_i)=1\) for all \(i\) because \(L\) is odd.
One can check that \(\ker D^*\) is nontrivial when \(L\) is even. 
Thus after we fix the normalization \(P(\vb{x}) =1\), there is a unique critical solution. 

Due to the structure of the cycle graph and our choice of constant vertex densities, the system of criticality equations is symmetric under the transitive action of the dihedral group of order \(2L\) on \(x_1, \dots, x_L\). This symmetry, combined with the uniqueness of the critical solution, implies that the solution satisfies \(\rmx_{1} =\cdots = \rmx_{L}\).
We then only have to solve for \(x_0\), \(x_1\) in the system
\[
\begin{aligned}
x_i \partial_{x_i} P(\vb*{x})\big \vert_{x_1=\cdots=x_L} &= \alpha & \text{for } i = 1,\dots, L\\
x_0 \partial_{x_0} P(\vb*{x})\big \vert_{x_1=\cdots=x_L} &= L(1-\alpha) \\
P(\vb*{x})\big \vert_{x_1=\cdots=x_L} &= 1.
\end{aligned} 
\]
Observe that the sum of the first \(L+1\) equations is \(L\) times the last equation, and the first \(L\) equations are identical. 
We also see that in this case, applying the degree operator \(x_0 \partial_{x_0}\) to \(P\) commutes with setting \(x_{1} = \cdots = x_{L}\).
Let 
\(P_0(x_0,x_1) \coloneq P(\vb*{x})\big \vert_{x_1=\cdots=x_L} \) 
be the \textit{reduced tiling polynomial}.
Now, finding the critical solution amounts to solving
\begin{align}
x_0 \partial_{x_0} P_0(x_0,x_1) &= L(1-\alpha) \label{eqn:reducedcritsyszero} \\
P_0(x_0,x_1) &= 1. \label{eqn:reducedcritsysnorm}
\end{align}
Then, \(\rmx_0, \rmx_1\) will be the roots of some polynomial depending on \(L\) and \(\alpha\). Applying Theorem \ref{thm:critgauge}, the exponential growth rate will be 
\(\sigma(w, \alpha)= - L(1-\alpha) \log (\rmx_0)-L \alpha \log (\rmx_1) \).

\subsubsection{Reduced tiling polynomial}

We find a closed-form expression for \(P_0\) by performing recursion on the size \(L\) of the graph. 

\begin{proposition}[Reduced tiling polynomial] \label{prop:redtilingpoly}
The reduced tiling polynomial for a cycle graph of length \(L \geq 1\) is
\begin{equation} \label{eqn:P0}
P^L_0(x_0,x_1) = 2^{-L} \pqty{\left(x_0 -\sqrt{x_0^2 + 4x_1^2} \right)^L + \left(x_0 +\sqrt{x_0^2 + 4x_1^2} \right)^L}.
\end{equation}
\end{proposition}

\begin{proof}
Let \(Q_0^L(x_0, x_1)\) denote the reduced tiling polynomial for a line graph with \(L\) vertices, where all tile weights are \(1\). Once we find \(Q_0^L\) we will be done since
\(P_0^L = Q_0^L + x_1^2 Q_0^{L-2}\).
The first summand corresponds to all partial matchings not containing the edge \(v_L v_1\), and the second summand corresponds to the case where this edge is present. We find \(Q_0^L \) recursively. We have \(Q_0^0 =1\) and \(Q_0^1 = x_0\), and in general 
\(Q_0^L = x_0 Q_0^{L-1} + x_1^2 Q_0^{L-2}\). On the right-hand side of this recurrence, the first (resp. second) term corresponds to the case where the first vertex is not covered (resp. is covered) by the partial matching. We then solve for \(Q_0^L\) and obtain
\begin{equation}\label{eqn:Q0L}
Q_0^L(x_0,x_1) = \frac{2^{-(L+1)}}{\sqrt{x_0^2 + 4x_1^2}} \pqty{-\left(x_0-\sqrt{x_0^2+4 x_1^2}\right)^{L+1}+\left(x_0 + \sqrt{x_0^2+4 x_1^2}\right)^{L+1}} .
\end{equation}
\end{proof}

Let \(\varphi = \frac{1+\sqrt{5}}{2}\) be the golden ratio and let \(\psi = - \frac{1}{\varphi} = \frac{1-\sqrt{5}}{2}\). We can enumerate the number of partial matchings of a line graph (resp. cycle graph) by substituting \(x_0 = 1\), \(x_1 = 1\) into the reduced tiling polynomial \(Q_0^L\) (resp. \(P_0^L\)).

\begin{corollary} \label{cor:numtiles} 
The number of partial matchings of a line graph of length \(L-1\) is the \(L\)-th Fibonacci number \(F_{L}=\frac{\varphi^L - \psi^L}{\sqrt{5}}\). 
%
The number of partial matchings of a cycle graph of length \(L\) is \(\varphi^L + \psi^L\).
\end{corollary}

\subsection{Uniform critical weights and critical vertex densities}

The critical tile weights are \(\rmx_0^{L-2s(t)} \rmx_1^{2s(t)}\), so they depend only on the number of edges in the partial matching corresponding to \(t\). Figure \ref{fig:L=11butterfly} is a plot of these tile probabilities as the vertex density \(\alpha\) varies, and there appears to be an \(\alpha\) for which all tile probabilities equal. In fact we can show that this special value of \(\alpha\) always exists.

\begin{figure} 
\centering
\includegraphics[scale=0.58]{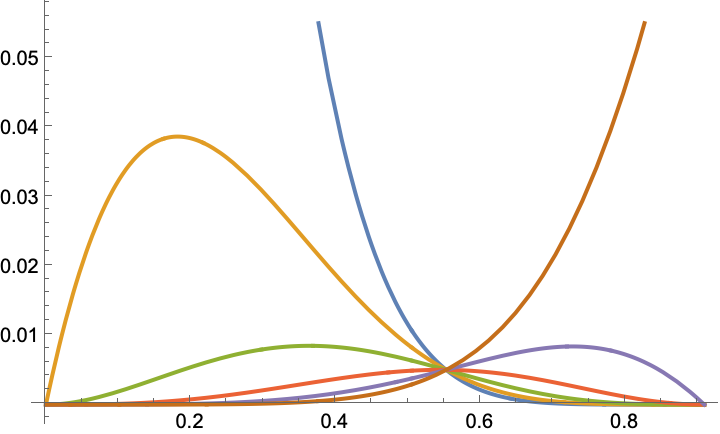}
\caption{Critical tile probabilities for \(L=11\) cycle graph and \(\alpha \in \left(0,\frac{10}{11}\right)\). The tile probability for tile sizes 0, 1, \textellipsis, 5 is blue, orange, green, red, purple, brown, respectively.}
\label{fig:L=11butterfly}
\end{figure}

\begin{theorem}[Critical density] \label{thm:alphahat}
For a cycle graph of any odd length \(L\geq3\), there exists a unique vertex density \(\hat{\alpha}\) for which the critical tile distribution is uniform, and
\begin{equation} \label{eqn:alphahat}
\hat \alpha = 1- \frac{(\varphi^L - \psi^L)/\sqrt{5}}{\varphi^L + \psi^L}.
\end{equation}  
Furthermore, \(\hat{\alpha}\) is the unique density which maximizes the exponential growth rate \(\sigma(\alpha)\).
\end{theorem}
\begin{proof} 
Recall that the critical solution \(\rmx_0\), \(\rmx_1\) is by definition positive (and such a solution always exists); see Theorem \ref{thm:critgauge}. 

We first show that \(\rmx_0\), \(\rmx_1\) are continuously differentiable functions of \(\alpha\) using the Implicit Function Theorem. Triples \(\alpha\), \(\rmx_0\), \(\rmx_1\) satisfying the reduced criticality equations (\ref{eqn:reducedcritsyszero}), (\ref{eqn:reducedcritsysnorm}) are exactly zeroes of the function \(f:\R^{1+2}\to \R^2\) given by
\[f(\alpha, x_0, x_1) = \pmqty{P_0 - 1\\ x_0 \partial_{x_0} P_0 - L(1-\alpha)}.\]
We verify that the Jacobian \(J_{f,(x_0,x_1)}\) is invertible at zeros of \(f\); the determinant simplifies to
\[\det J_{f,(x_0,x_1)} = \frac{-4L^2 \pqty{-L x_0 x_1^{2L} + x_1^2 Q_0^{2L-1}}}{x_1 (x_0^2 + 4x_1^2 )},\]
where \(Q_0^{2L-1}\) is the reduced tiling polynomial from (\ref{eqn:Q0L}). 
The coefficient \([x_0 x_1^{2L-2}] Q_0^{2L-1}\) is \(2L-1\), so one can see that \(-L x_0 x_1^{2L} + x_1^2 Q_0^{2L-1}\) is a subtraction-free polynomial in \(x_0\), \(x_1\). Since \(\rmx_0\), \(\rmx_1 \) are positive, we have \(\det J_{f, (x_0,x_1)} <0\) at any zero \((\alpha, \rmx_0, \rmx_1)\) of \(f\). Thus, \(\rmx_0(\alpha)\) and \(\rmx_1(\alpha)\) are continuously differentiable.

Recalling that \(\alpha\) is the asymptotic fraction of tiles covering each vertex \(v_1,\dots, v_L\), we have \(\rmx_0 \to 1\) and \(\rmx_1\to 0\) as \(\alpha \to 0\), and \(\rmx_0\to 0\) and \(\rmx_1 \to \infty\) as \(\alpha \to \frac{L-1}{L}\). If we can show that \(\rmx_0(\alpha)\) and \(\rmx_1(\alpha)\) are strictly monotone on \(\pqty{0, \frac{L-1}{L} }\), then there will be a unique vertex density, denoted by \(\hat{\alpha}\), such that \(\rmx_0(\hat{\alpha})=\rmx_1(\hat{\alpha})\). This condition is equivalent to the critical tile weights \(\rmx_0^{L-2s(t)}\rmx_1^{2s(t)}\) being uniform, since \(\rmx_0, \rmx_1 > 0\). 

With \(x_i =\rmx_i(\alpha)\) and expressing the reduced tiling polynomial as a sum over tiles 
\(P_0(x_0, x_1) = \sum_t x_0^{L-2 s(t)} x_1^{2s(t)}\), 
we may differentiate both sides of equations (\ref{eqn:reducedcritsyszero}) and (\ref{eqn:reducedcritsysnorm}) with respect to \(\alpha\) to get
\begin{align}
\left( \sum_t (L-2s(t))^2 \rmx_0^{L-2s(t)-1} \rmx_1^{2s(t)}\right) \dv{\rmx_0}{\alpha}+ \left( \sum_t 2s(t) (L-2s(t))\rmx_0^{L-2s(t)} \rmx_1^{2s(t)-1} \right) \dv{\rmx_1}{\alpha}
&= -L \label{eqn:dalpha1} \\
\left( \sum_t (L-2s(t)) \rmx_0^{L-2s(t)-1} \rmx_1^{2s(t)}\right) \dv{\rmx_0}{\alpha}+ \left( \sum_t 2s(t) \rmx_0^{L-2s(t)} \rmx_1^{2s(t)-1} \right) \dv{\rmx_1}{\alpha}
&=0. \label{eqn:dalpha2} 
\end{align}
For a fixed \(\alpha\), the left-hand sides of the equations above are linear combinations of \(\dv{\rmx_0}{\alpha}\) and \(\dv{\rmx_1}{\alpha}\) with positive coefficients. 
Equation (\ref{eqn:dalpha2}) implies that \(\dv{\rmx_0}{\alpha}\) and \(\dv{\rmx_1}{\alpha}\) have opposite sign or are both zero; equation (\ref{eqn:dalpha1}) disallows the latter case. 
By our earlier reasoning about the behavior of \(\rmx_0\) and \(\rmx_1\) as \(\alpha\) approaches \(0\) or \(\frac{L-1}{L}\), and due to \(\dv{\rmx_0}{\alpha}\) and \(\dv{\rmx_1}{\alpha}\) being continuous, we may conclude that these derivatives are strictly negative and positive, respectively, for all \(\alpha\).

To find the value of \(\hat{\alpha}\), the normalization (\ref{eqn:reducedcritsysnorm}) combined with the fact that \(\rmx_0 (\hat{\alpha})=\rmx_1(\hat{\alpha})\) gives
\(P_0(x_0,x_1) = \sum_t \rmx_0(\hat{\alpha})^L = 1\).
Then, \(\rmx_0 (\hat{\alpha})=\rmx_1(\hat{\alpha}) =\abs{T}^{-\frac1L}\), where the number of tiles \(\abs{T}\) is \(\varphi^L + \psi^L\) by Corollary \ref{cor:numtiles}. 
Write equation (\ref{eqn:reducedcritsyszero}) in closed form using Proposition \ref{prop:redtilingpoly}:
\[ \frac{-L \rmx_0}{\sqrt{\rmx_0^2 + 4\rmx_1^2}} \left(2^{-L}\right) \pqty{\left(\rmx_0 -\sqrt{\rmx_0^2 + 4\rmx_1^2} \right)^L + \left(\rmx_0 +\sqrt{\rmx_0^2 + 4\rmx_1^2} \right)^L}= L(1-\alpha).\]
We evaluate both sides above at \(\alpha = \hat{\alpha}\) (so in particular \(\rmx_i (\hat{\alpha}) = \abs{T}^{-\frac1L}\)). Then, isolating \(\hat{\alpha}\) gives the result. 

Lastly, we wish to show that \(\hat{\alpha}\) maximizes the growth rate uniquely. We established earlier in Section \ref{subsec:growthrate} that \(\sigma\) is a strictly concave function of \(\alpha\), so all we have to show is that \(\hat{\alpha}\) is a critical point. We have
\begin{equation} \label{}
\dv{\sigma}{\alpha} = L \log \left( \frac{\rmx_0}{\rmx_1}\right) - \frac{L(1-\alpha)}{\rmx_0} \dv{\rmx_0}{\alpha}  - \frac{L \alpha}{\rmx_1}  \dv{\rmx_1}{\alpha}.
\end{equation}
Another expression for \(\hat{\alpha}\) coming from (\ref{eqn:reducedcritsyszero}) and (\ref{eqn:reducedcritsysnorm}) is \(\hat{\alpha} = \frac{2}{L \abs{T}} \sum_t s(t)\).
Then, the left-hand side of (\ref{eqn:dalpha2}) evaluated at \(\hat{\alpha}\) is \(-\dv{\sigma}{\alpha} (\hat{\alpha}){(\hat{\alpha})}\), and so \(\dv{\sigma}{\alpha}{(\hat{\alpha})}=0\). 
\end{proof}

We can make sense of the formula for \(\hat{\alpha}\) by asking: when all tile types are equally likely, what is the asymptotic fraction of tiles covering a fixed vertex \(v\)? In this case, it is the same as the fraction of tile types which cover \(v\). Denote the number of partial matchings of a cycle (resp. line) graph with \(L\) vertices by \(Y_L\) (resp. \(W_L\)). The tile types not covering \(v\) are in bijection with partial matchings of a line graph with \(L-1\) vertices, so \(\hat{\alpha} = \frac{Y_L - W_L}{Y_L} = 1 - \frac{W_{L-1}}{Y_L}\).
From here on we will focus our investigation on the situation of \(\alpha_{v_1} = \dots = \alpha_{v_L} = \hat{\alpha}\) and uniform tile probabilities. 

\subsection{Inverse Laplacian} \label{subsec:invlaplac}

By Theorem \ref{thm:Xtdistr}, the random vector counting tile multiplicities \(\vb{X}\) has a multidimensional Gaussian distribution when \(n_v\) and \(N\) are large; its mean and covariance matrix are
\begin{equation}\label{eqn:Xtmeancov}
\Expe(X_t) = \frac{N}{\abs{T}} \qand \Cov(\vb{X}) = \frac{N}{\abs{T}} \pqty{I -D^\top (D D^\top)^{-1} D}. 
\end{equation}
See Figure \ref{fig:L=9covX} for an instance of this covariance matrix. 

The main ingredient for computing covariances concretely from (\ref{eqn:Xtmeancov}) is inverting the Laplacian \(\Delta = DD^\top\) (in principle \(\Delta\) should contain a factor of \(\frac{1}{\abs{T}}\), but for convenience we elide it because it cancels out in (\ref{eqn:Xtmeancov})). This matrix is symmetric by definition, and we will see that it mostly consists of a circulant submatrix. The eigenvalues and normalized eigenvectors of a circulant matrix are known---we refer the reader to any reference on Toeplitz matrices (e.g. \cite{G06}) for details.

The \((i,j)\)-entry of \(\Delta\) is the number of partial matchings which contain both vertex \(v_i\) and vertex \(v_j\), counted with multiplicity \(t_{v_i} t_{v_j}\). We can write this quantity in terms of the tiling polynomial:
\[x_{i} \partial_{x_i} x_j \partial_{x_j} P(\vb*{x}) \vert_{x_i=1} = \sum_t t_{v_i} t_{v_j}.\]
We compute the entries in the zeroth row and column using the closed-form expression for \(P_0(x_0,x_1)\) from (\ref{eqn:P0}). Firstly,
\[\Delta_{00} = x_{0} \partial_{x_0} x_0 \partial_{x_0} P_0(x_0,x_1) \vert_{x_i=1}  = \frac{L}{5} \pqty{4 F_L + L \frac{F_{2L}}{F_L} }.\]
Applying a symmetry argument among the vertices \(v_1,\dots, v_L\), we also have
\[\Delta_{0v} = \Delta_{v0} = \frac1L x_{0} \partial_{x_0} x_1 \partial_{x_1} P_0(x_0,x_1) \vert_{x_i=1} = -\frac15 \pqty{(4-5L)F_L + L \frac{F_{2L}}{F_L}} \qc v = v_1,\dots, v_L. \]

Now consider the case where neither \(v_i\) nor \(v_j\) are the zero vertex. 
We count the number of partial matchings containing both vertices by considering the four pairs of line graphs coming from cutting \(\calG\) before or after \(i\) and \(j\), so this total depends on the distance \(\abs{i-j}\). From Corollary \ref{cor:numtiles}, we know that partial matchings on a line graph are enumerated by the Fibonacci sequence.
Then, the \((L\times L)\)-submatrix indexed by vertices \(v_1,\dots, v_L\) is circulant with entries \(\Delta_{ij} = c_{\abs{i-j}}\) where
\[
c_0 = 2 F_{L-1} \qand c_k = F_k F_{L-k-2} + 2F_{k-1} F_{L-k-1} + F_{k-2} F_{L-k} \qc k=1,\dots, L-1.
\]
We can now write down the \(L\) eigenvalues and orthonormal eigenvectors of the circulant submatrix. 
Let \(\omega = e^{\frac{2 \pi i}{L}}\). For \(k=0,\dots, L-1\), the eigenvalues are 
\begin{equation} \label{eqn:eigenvals2}
\lambda_k = c_0 + c_1 \omega^k + \cdots + c_{L-1} (\omega^k)^{L-1}= F_L \frac{(1+\omega^k)^2}{1+3\omega^k + \omega^{2k} },
\end{equation}
which correspond to the eigenvectors 
\[u_k=L^{-\frac12} \pqty{1, \omega^k, \cdots, (\omega^k)^{L-1}} ^\top.\]

Given the special form of the matrix, one can manually modify these eigenvectors as well as find one additional orthonormal eigenvector to obtain a complete eigensystem for \(\Delta\). We then invert \(\Delta\) via its eigendecomposition. In doing so, the following sum over roots of unity arises in entries of the matrix; let 
\begin{equation} \label{eqn:gLldef}
g_L(l)= \sum_{k=1}^{L-1} \frac{\omega^{kl}}{\lambda_k}.
\end{equation}
We also define the \((L\times L)\)-matrix
\begin{equation} \label{eqn:LxLsubmatrix}
A =  \frac{1}{L} \pqty{ \frac{\Delta_{00}}{\lambda_0 \Delta_{00} - L \Delta_{0v}^2} + g_L(i-j) }_{ i,j=0,\dots, L-1}.  
\end{equation} 
As a block matrix, the inverse of the tiling Laplacian is then
\begin{equation} \label{eqn:invDelta}
\Delta^{-1} 
= 
\setlength{\extrarowheight}{-4pt}
\renewcommand{\arraystretch}{2}
\pmqty{
\mqty{\frac{\lambda_0}{\lambda_0 \Delta_{00} - L \Delta_{0v}^2}} & \rvline& 
\mqty{-\frac{ \Delta_{0v}}{\lambda_0 \Delta_{00} - L \Delta_{0v}^2}& \cdots &-\frac{ \Delta_{0v}}{\lambda_0 \Delta_{00} - L \Delta_{0v}^2}} \\ 
\hline
\mqty{-\frac{ \Delta_{0v}}{\lambda_0 \Delta_{00} - L \Delta_{0v}^2}\\ \vdots \\ -\frac{ \Delta_{0v}}{\lambda_0 \Delta_{00} - L \Delta_{0v}^2}} & \rvline &
\bigA
}
.
\end{equation}

\subsubsection{Sum over roots of unity}

The sum involving roots of unity in (\ref{eqn:gLldef}) may be written as a closed-form expression. First, let \(
a_l = \sum_{z^L = 1} \frac{z^l}{(z+1)^2}
\),
which is well defined since we assumed that \(L\) is odd, and so \(z\neq -1\).
Using (\ref{eqn:eigenvals2}) we have
\begin{equation}\label{eqn:gLlexpanded}
g_L(l)
=\frac{1}{F_L} \pqty{- \frac54 + \sum_{k=0}^{L-1} \frac{\omega^{kl} (1+3 \omega^k + \omega^{2k})}{(1+ \omega^k)^2} }
= \frac{1}{F_L} \pqty{- \frac54+ a_l + 3a_{l+1} + a_{l+2} }. 
\end{equation}
The following lemma gives a closed-form expression for \(a_l\); one can prove the lemma via recursion.
\begin{lemma} \label{lem:bl}
If \(L\geq1\) is odd, then 
\[\sum_{z^L = 1} \frac{z^l}{(z+1)^2}= \begin{cases} 
 -\frac{L}{4}(L-2) & l =0 \\
\frac{(-1)^{l+1} L}{4}  (L - 2l +2) & l=1,\dots, L-1.
\end{cases}\]
\end{lemma}
We may therefore conclude that 
\begin{equation} \label{eqn:gLl}
g_L(l) = \begin{cases}
\frac{1}{4 F_L} (L(L + 4) - 5)& l = 0 \\
\frac{1}{4 F_L}((-1)^l L (L-2l)- 5) & l=1,\dots, L-1.
\end{cases}
\end{equation}
Finally, since \(g_L(l)\) and \(a_l\) are both \(L\)-periodic in \(l\), we may extend periodically to obtain \(g_L(i-j)\) for each \(i,j = 0,\dots, L-1\), as required.

%% file: partialnfold_local_final.tex

\section{Local behavior on a cycle graph} \label{sec:exlocal}

Continuing with the same setup as in the previous section, we examine tile correlations when restricted to a window of five consecutive vertices of \(\calG\) (assume \(L >> 5\)). Specifically, we refer to a subset of \(\{v_L v_1, v_1 v_2, v_2 v_3, v_3 v_4, v_4 v_5, v_5 v_6\}\) covering each of \(v_1,\dots,v_5\) at most once as a local configuration; there are exactly 21 of them. Define an equivalence relation \(\sim\) on \(T\) where two tiles are equivalent if they use the same subset of edges on \(v_1,\dots, v_5\). The set of local configurations is then \(J = T / \sim\). Extend the quotient map linearly to obtain \(B: \R^T \to \R^J\). We also denote the matrix representing this map by \(B\). 

Define a random variable \(S_j: \Omega_{\vb{n},N} \to \Z_{\geq 0}\) which counts the number of times a tile with local configuration \(j\) occurs in a colored \(\vb{n}\)-multiweb, i.e. \(S_j(m) = \sum_{t\in j} m_t\).
The distribution of \(S_j\) will be determined by two parameters: let \(f_j\) be the number of edges in configuration \(j\) and let \(\varepsilon_j\) be how many of the endmost edges \(v_L v_1\), \(v_5 v_6\) are used. 
The random vector \(\vb{S} = (S_j)_j\) is a linear transformation of the random vector of tile multiplicities: \(\vb{S} = B\vb{X}\), and in particular \(S_j = \sum_{t\in j} X_t\). The cardinality of the equivalence class \(j\) is the number of partial matchings on the line graph of length \(L-\varepsilon_j-5\), which is \(F_{L-\varepsilon_j - 4}\) by Corollary \ref{cor:numtiles}.
Therefore from (\ref{eqn:Xtmeancov}) we conclude that when \(n_v\) and \(N\) are large, \(\vb{S}\) is a multidimensional Gaussian with
\begin{equation} \label{eqn:Sjmeancov}
\Expe(S_j) = \frac{N}{\abs{T}}F_{L-\varepsilon_j-4} 
\qand 
\Cov(\vb{S}) = \frac{N}{\abs{T}} \pqty{B B^\top - B D^\top (D D^\top)^{-1} D B^\top }.
\end{equation}
See Figure \ref{fig:L=31covS} for an example of this covariance matrix. 

\subsection{Limiting local tile distribution} \label{subsec:limlocal}

It turns out that the mean and variance of each \(S_j\) approach some limit as the (odd) size \(L\) of the graph tends to infinity. 
We have
\[\lim_{L\to \infty} \Expe(S_j) =\frac{N \varphi^{-\varepsilon_j-4}}{\sqrt{5} }
\qand
\lim_{L\to \infty} \Var(S_j) = \frac{N \varphi^{-\varepsilon_j-4}}{\sqrt{5} } \pqty{1 - \frac{ \varepsilon_j + \varphi^{-1} f_j + \frac{11+\sqrt{5}}{2} }{\varphi^{\varepsilon_j + 6}}} .\]
From (\ref{eqn:Sjmeancov}), we can take the limit of the mean directly. 
For the variance, one can for example obtain an expanded expression for \(e_j^\top \Cov(\vb{S}) e_j\) via computer using the explicit descriptions of \(B\) and \(D\), and of \(\Delta^{-1}\) in (\ref{eqn:LxLsubmatrix}), (\ref{eqn:invDelta}), (\ref{eqn:gLl}). Then, manually identifying the terms which dominate asymptotically gives the result.